\documentclass[a4paper, 11pt]{amsart}
\usepackage[utf8]{inputenc}
\usepackage[T1]{fontenc}
\usepackage[english]{babel}
\usepackage{graphicx}
\usepackage{stmaryrd}
\usepackage{tikz}
\usepackage{tikz-cd} 
\usepackage[all]{xy}
\usepackage{comment}
\usepackage{amsmath}
\usepackage{systeme}
\usepackage{bm}
\usepackage{amsmath,amsfonts,amssymb}
\usepackage[a4paper]{geometry}
\usepackage{amsthm}
\usepackage{float}
\usepackage{enumitem}
\usepackage{hyperref}
\usepackage{xcolor}
\hypersetup{
    colorlinks=true,
    linkcolor={blue},
    citecolor={blue},
    urlcolor={blue}}
\newtheorem{te}{Theorem}[section]

\newtheorem{co}[te]{Corollary}

\newtheorem{qu}[te]{Question}
\newtheorem{lemme}[te]{Lemma}
\theoremstyle{definition}

\newtheorem{de}[te]{Definition}

\theoremstyle{remark}
\newtheorem{rque}[te]{Remark}
\newlength{\plarg}
\usetikzlibrary{cd}
\calclayout

\title{Acylindrical hyperbolicity and existential closedness}
\author{Simon André}
\begin{document}

\begin{minipage}{\linewidth}

\maketitle

\vspace{0mm}

	\begin{abstract}
Let $G$ be a finitely presented group, and let $H$ be a subgroup of $G$. We prove that if $H$ is acylindrically hyperbolic and existentially closed in $G$, then $G$ is acylindrically hyperbolic. As a corollary, any finitely presented group which is existentially equivalent to the mapping class group of a surface of finite type, to $\mathrm{Out}(F_n)$ or $\mathrm{Aut}(F_n)$ for $n\geq 2$ or to the Higman group, is acylindrically hyperbolic.
\end{abstract}
	
\end{minipage}

\thispagestyle{empty}

\vspace{5mm}

\section{Introduction}

Acylindrically hyperbolic groups, defined by Osin in \cite{Osi16}, form a large class of groups that has received a lot of attention in the recent years. Notable examples of acylindrically hyperbolic groups include non-elementary hyperbolic and relatively hyperbolic groups, mapping class groups of most surfaces of finite type, $\mathrm{Out}(F_n)$ and $\mathrm{Aut}(F_n)$ for $n\geq 2$, many groups acting on $\mathrm{CAT}(0)$ spaces, and many other groups. This short note is motivated by the following question, asked by Osin.

\begin{qu}\label{osin}Is acylindrical hyperbolicity preserved under elementary equivalence among finitely generated groups?
\end{qu}

In \cite{Sel09}, Sela proved that hyperbolicity is preserved under elementary equivalence, among torsion-free finitely generated groups, and we proved that the torsion-freeness assumption can be omitted (see \cite{And18}). The question whether there exists an analogous result for weaker forms of negative curvature in groups is natural. Note that unlike hyperbolic groups, acylindrically hyperbolic groups need not be finitely generated, and Question \ref{osin} makes sense even if we don't assume finite generation; however, the answer to this question is negative in general, even among countable groups (see Section \ref{counter} for further details).

\smallskip

Let $G$ be a group, and let $H$ be a subgroup of $G$. One says that $H$ is \emph{existentially closed} in $G$ if the following holds: for every existential formula $\phi(\bm{x})$ and every tuple $\bm{h}\in H^p$ with $p=\vert \bm{x}\vert$, the sentence $\phi(\bm{h})$ is satisfied by $H$ if and only if it is satisfied by $G$. Equivalently, $H$ is existentially closed in $G$ if any disjunction of systems of equations and inequations with constants in $H$ has a solution in $H$ if and only if it has a solution in $G$. We prove the following result.


\begin{te}\label{above}
Let $G$ be a finitely presented group, and let $H$ be a subgroup of $G$. If $H$ is acylindrically hyperbolic and existentially closed in $G$, then $G$ is acylindrically hyperbolic.\end{te}

A group $G$ is called \emph{equationally Noetherian} if the set of solutions in $G$ of any infinite system of equations in finitely many variables $\Sigma$ coincides with the set of solutions in $G$ of a finite subsystem of $\Sigma$. For instance, hyperbolic groups are equationally Noetherian (see \cite{Sel09} and \cite{RW14}). By contrast, acylindrically hyperbolic groups are not equationally Noetherian in general: given a group $H$ that is not equationally Noetherian, the free product $H\ast\mathbb{Z}$ is acylindrically hyperbolic and is not equationally Noetherian (since the equational Noetherian property is inherited by subgroups). We prove the following variant of Theorem \ref{above}.

\begin{te}\label{above2}
Let $G$ be a finitely generated group, and let $H$ be a subgroup of $G$. If $H$ is acylindrically hyperbolic, equationally Noetherian and existentially closed in $G$, then $G$ is acylindrically hyperbolic.
\end{te}

\begin{rque}
Note that, in Theorems \ref{above} and \ref{above2}, the subgroup $H$ is not assumed to be finitely generated.
\end{rque}

These results follow from the well-known Rips machine, adapted by Groves and Hull to acylindrically hyperbolic groups in \cite{GH19}.

\smallskip

Note that Theorems \ref{above} and \ref{above2} do not remain true if one removes the assumption that the subgroup $H$ is existentially closed. Indeed, acylindrical hyperbolicity is not inherited by overgroups in general. As an example, $H=F_2$ is acylindrically hyperbolic, but the group $G=F_2\times \mathbb{Z}$ is not acylindrically hyperbolic. In this example, one easily sees that $H=F_2$ is not existentially closed in $G$: let $h_1$ and $h_2$ be two non-commuting elements of $H$. The only element of $H$ that commutes both with $h_1$ and $h_2$ is the trivial element. Hence, the following system of equations and inequations has a solution in $G$ but not in $H$: \[([x,h_1]=1)\wedge ([x,h_2]=1)\wedge (x\neq 1).\]

\smallskip

In addition, we construct in Section \ref{counter} an elementary embedding of the free group $F_2$ into a (necessarily not finitely generated) group that is not acylindrically hyperbolic, which shows that the property that $G$ is finitely generated cannot be omitted from Theorem \ref{above2}, even if we replace the condition that $H$ is existentially closed with the stronger condition that $H$ is elementarily embedded into $G$.



\smallskip

One says that two groups $G$ and $H$ are \emph{existentially equivalent} if they satisfy the same existential first-order sentences. In general, acylindrical hyperbolicity is not preserved under existential equivalence, even among finitely presented groups. For instance, $H=F_2\times \mathbb{Z}$ is existentially closed in $G=H\ast\mathbb{Z}$ as a consequence of Lemma \ref{lemme1} below (since there exists a discriminating sequence of retractions $(\varphi_n : G \rightarrow H)_{n\in \mathbb{N}}$). In particular, $H$ and $G$ are existentially equivalent. But $G$ is acylindrically hyperbolic and $H$ is not acylindrically hyperbolic. However, we deduce the following result from Theorem \ref{above} (see Section \ref{section_coro} for details).

\begin{co}\label{coro}
Let $G$ be a finitely presented group. If $G$ is existentially equivalent to the mapping class group $\mathrm{Mod}(\Sigma_g)$ of a closed orientable surface $\Sigma_g$ of genus $g\geq 4$, then there is an embedding $i : \mathrm{Mod}(\Sigma_g) \hookrightarrow G$ such that $i(\mathrm{Mod}(\Sigma_g))$ is existentially closed in $G$. Therefore, $G$ is acylindrically hyperbolic (by Theorem \ref{above}). The same result is true if one replaces $\mathrm{Mod}(\Sigma_g)$ with one of the following groups: $\mathrm{Out}(F_n)$ and $\mathrm{Aut}(F_n)$ for $n\geq 2$, or the Higman group.\end{co}

\smallskip

The acylindrical hyperbolicity of $\mathrm{Mod}(\Sigma_g)$ was proved in \cite{Bow08}. For $\mathrm{Out}(F_n)$, $\mathrm{Aut}(F_n)$ and the Higman group, acylindrical hyperbolicity was proved respectively in \cite{BF09}, \cite{GH20} and \cite{Mar17}. Recall that the Higman group was constructed in \cite{Hig51} as the first example of a finitely presented infinite group without non-trivial finite quotients. It is given by the following presentation: \[H=\langle a_1,a_2,a_3,a_4 \ \vert \ a_ia_{i+1}a_i^{-1}=a_{i+1}^2, \ i\in \mathbb{Z}/4\mathbb{Z}\rangle.\]

\smallskip

In \cite{GH19}, in the paragraph below Definition 3.1, Groves announced that the mapping class group of a surface of finite type is equationally Noetherian. By Theorem \ref{above2}, this result implies that one can replace `finitely presented' with `finitely generated' in the previous statement.

\begin{co}
Let $G$ be a finitely generated group. If $G$ is existentially equivalent to $\mathrm{Mod}(\Sigma_g)$ for $g\geq 4$, then $G$ is acylindrically hyperbolic.
\end{co}

Last, it is worth mentioning a result proved recently by Bogopolski in \cite{Bog18}. A subgroup $H$ of a group $G$ is said to be \emph{verbally closed} if, for any $w(x_1,\ldots,x_n)\in F_n$ and $h\in H$, the equation $w(x_1,\ldots,x_n)=h$ has a solution in $H$ if and only if it has a solution in $G$. Bogopolski proved the following theorem: let $G$ be a finitely presented group, and let $H$ be a finitely generated, acylindrically hyperbolic subgroup of $G$. Suppose in addition that $H$ has no non-trivial normal finite subgroup. Then $H$ is verbally closed in $G$ if and only if $H$ is a retract of $G$. Moreover, if $H$ is equationally Noetherian, one can simply assume that $G$ is finitely generated. Of course, this result does not imply that $G$ is acylindrically hyperbolic, since for instance $F_2$ is a retract of $F_2\times\mathbb{Z}$, which is not acylindrically hyperbolic.



\subsection*{Acknowledgments}I warmly thank Denis Osin for very useful comments on an earlier version of this paper.

\section{Acylindrically hyperbolic groups}

The following definition was introduced by Bowditch in \cite{Bow08}.

\begin{de}An action of a group $G$ by isometries on a metric space $(X,d)$ is called \emph{acylindrical} if for every $\varepsilon \geq 0$ there exist $N>0$ and $R>0$ such that for every two points $x,y \in X$ satisfying $d(x,y) \geq R$, there are at most $N$ elements $g\in G$ such that
    \begin{align*}
        d(x,gx)\leq \varepsilon \ \text{ and } \ d(y,gy)\leq \varepsilon.
    \end{align*}
    \end{de}
    
We recall the following classical definitions.

\begin{de}Let $(X,d)$ be a $\delta$-hyperbolic metric space, and let $G$ be a group acting on $X$ by isometries. An element $g\in G$ is called \emph{elliptic} if some (equivalently, any) orbit of $g$ is bounded, and \emph{loxodromic} if the map $\mathbb{Z}\rightarrow X : n \mapsto g^nx$ is a quasi-isometry for some (equivalently, any) $x\in X$. Every loxodromic element $g\in G$ has exactly two limit points $g^{+\infty}$ and $g^{-\infty}$ on the Gromov boundary $\partial_{\infty}X$. Two loxodromic elements $g,h\in G$ are called \emph{independent} if the sets $\lbrace g^{\pm\infty}\rbrace$ and $\lbrace h^{\pm\infty}\rbrace$ are disjoint.\end{de} 

If a group $G$ admits an acylindrical action on a $\delta$-hyperbolic metric space, then $G$ satisfies one of the following three conditions (see \cite[Theorem 1.1]{Osi16}).
\begin{enumerate}
    \item $G$ is elliptic, that is every element $g\in G$ is elliptic.
    \item $G$ is virtually cyclic and contains a loxodromic element.
    \item $G$ contains two (equivalently, infinitely many) pairwise independent loxodromic elements. In this case, one says that $G$ is \emph{acylindrically hyperbolic}.
\end{enumerate}
In the first two cases, one says that the action of $G$ is \emph{elementary}. Note that every group has an elementary acylindrical action on a hyperbolic space, namely the trivial action on a point. For this reason, the third condition is the only one of interest; in this case, one says that the action of $G$ is \emph{non-elementary}.

\section{Proof of Theorems \ref{above} and \ref{above2}}


Given two groups $G$ and $H$, a sequence of morphisms $(\varphi_n)_{n\in\mathbb{N}}\in\mathrm{Hom}(G,H)^{\mathbb{N}}$ is said to be \emph{discriminating} if the following condition holds: for every finite subset $B\subset G\setminus\lbrace 1\rbrace$, there exists an integer $n_B$ such that $\ker(\varphi_n)\cap B=\varnothing$ for every $n\geq n_B$.

\begin{lemme}\label{lemme1}Let $G$ be a finitely presented group, and let $H$ be a subgroup of $G$. The following assertions are equivalent.
\begin{enumerate}
    \item $H$ is existentially closed in $G$.
    \item For every finitely generated subgroup $H'\subset H$, there exists a discriminating sequence $(\varphi_n : G\twoheadrightarrow H)_{n\in\mathbb{N}}$ such that ${\varphi_n}_{\vert H'}=\mathrm{id}_{H'}$.
\end{enumerate}
\end{lemme}

\begin{proof}We first prove that \emph{(2)} implies \emph{(1)}. Let $\theta(\bm{x},\bm{h})$ be a quantifier-free formula with constants from $H$. Let $H'$ be the subgroup of $H$ generated by $\bm{h}$. By assumption, there exists a discriminating sequence $(\varphi_n : G\twoheadrightarrow H)_{n\in\mathbb{N}}$ such that ${\varphi_n}_{\vert H'}=\mathrm{id}_{H'}$. Suppose that there exists a tuple $\bm{g}$ of elements of $G$ such that $\theta(\bm{g},\bm{h})$ holds in $G$. For $n$ sufficiently large, $\theta(\varphi_n(\bm{g}),\bm{h})$ holds in $H$.

\smallskip

Now, let us assume that $H$ is existentially closed in $G$. Let $\langle \bm{s} \ \vert \ R(\bm{s})=1\rangle$ be a finite presentation of $G$. Let $H'$ be a finitely generated subgroup of $H$ and let $\lbrace h_1,\ldots, h_k\rbrace$ be a generating set of $H'$. We denote by $B_n=\lbrace b_1,\ldots,b_{N(n)}\rbrace\subset G$ the ball of radius $n$ for the metric induced by $\bm{s}$, with $b_1=1$. Every element $h_i$ (resp.\ $b_j$) can be written as a word $w_i(\bm{s})$ (resp.\ $v_j(\bm{s})$). The following system of equations and inequations over $H$ has a solution in $G$, namely $\bm{x}=\bm{s}$: \[(R(\bm{x})=1)\wedge\left(\bigwedge_{i=1}^kh_i=w_i(\bm{x})\right)\wedge\left(\bigwedge_{j=2}^{N(n)}b_j(\bm{x})\neq 1\right).\]Since $H$ is existentially closed in $G$, this system has a solution $\bm{t}$ in $H$ as well. The morphism $\varphi_n : G\rightarrow H : \bm{s}\mapsto \bm{t}$ is well-defined, does not kill any element of $B_n\setminus\lbrace 1\rbrace$, and maps $h_i=w_i(\bm{s})$ to $h_i=w_i(\bm{t})$.\end{proof}

The following lemma can be proved in a similar way. Indeed, given a presentation of $G$ of the form $\langle \bm{s} \ \vert \ R(\bm{s})=1\rangle$, with $R$ possibly infinite, the equational Noetherian property satisfied by $H$ has the following consequence: there exists a finite subset of relations $R_0\subset R$ such that $\mathrm{Hom}(G,H)$ is in one-to-one correspondence with the set of solutions in $H$ of the system of equations $R_0(\bm{x})=1$.

\begin{lemme}\label{lemme2}Let $G$ be a finitely generated group, and let $H$ be an equationally Noetherian subgroup of $G$. The following assertions are equivalent.
\begin{enumerate}
    \item $H$ is existentially closed in $G$.
    \item For every finitely generated subgroup $H'\subset H$, there exists a discriminating sequence $(\varphi_n : G\twoheadrightarrow H)_{n\in\mathbb{N}}$ such that ${\varphi_n}_{\vert H'}=\mathrm{id}_{H'}$.
\end{enumerate}
\end{lemme}

We are ready to prove Theorems \ref{above} and \ref{above2}. These results follow from Groves' and Hull's paper \cite{GH19}, in which the authors generalised the well-known Rips machine to acylindrically hyperbolic groups. 


\begin{proof}[Proof of Theorem \ref{above} and Theorem \ref{above2}]
Let $G$ and $H$ be two groups as in Theorem \ref{above} or Theorem \ref{above2}. Let $S$ be a finite generating set of $G$. Let $(X,d)$ be a $\delta$-hyperbolic metric space on which the group $H$ acts acylindrically and non-elementarily. Let $h_1$ and $h_2$ be two independent loxodromic elements of $H$. By Lemma \ref{lemme1} or Lemma \ref{lemme2}, there exists a discriminating sequence $(\varphi_n : G\twoheadrightarrow H)_{n\in\mathbb{N}}$ such that ${\varphi_n}(h_1)=h_1$ and ${\varphi_n}(h_2)=h_2$. Let $\omega$ be a non-principal ultrafilter. We define the \emph{scaling factor} of the homomorphism $\varphi_n$ as follows: \[\lambda_n=\inf_{x\in X}\max_{s\in S}d(x,\varphi_n(s)x).\]We denote by $\lambda$ the $\omega$-limit of the sequence $(\lambda_n)_{n\in\mathbb{N}}$, and distinguish two cases.

\smallskip

\textbf{First case.} Suppose that $\lambda=+\infty$. Then the asymptotic cone \[(X_{\omega},d_{\omega})=\left(\prod_{n \in \mathbb{N}} (X,d/\lambda_n)\right) / \omega\] is a real tree, and $G$ acts on this tree non-trivially by isometries (see for instance \cite[Theorem 4.4]{GH19} for details). Let $T$ be a minimal and $G$-invariant subtree of $X_{\omega}$. The action of $G$ on this tree can be analysed using the Rips machine, adapted by Groves and Hull in \cite{GH19} to the setting of acylindrically hyperbolic groups. The Rips machine converts the action of $G$ on $T$ into a non-trivial splitting $G=A\ast_C B$ or $G=A\ast_C$ where $C$ is a virtually abelian group (for details, we refer the reader to \cite[Theorem 4.18, Proposition 4.19 and Lemma 5.1]{GH19}). If $C$ is finite, then $G$ is acylindrically hyperbolic (see for instance \cite[Corollaries 2.2 and 2.3]{MO15}). Now, assume that $C$ is infinite. By \cite[Lemma 5.6]{GH19}, there exists a unique maximal virtually abelian group $M$ containing $C$. This group is defined as follows: \[M=\langle \lbrace g\in G \ \vert \ \langle g,C\rangle \text{ is virtually abelian}\rbrace\rangle.\]Moreover, by \cite[Lemma 5.7]{GH19}, the group $M$ has the following key property: for every $g\in G\setminus M$, the intersection $M\cap gMg^{-1}$ is finite. As a consequence, there exists an element $g\in G$ such that the intersection of $gCg^{-1}$ and $C$ is finite (one says that $C$ is \emph{weakly malnormal} in $G$); otherwise, $M\cap gMg^{-1}$ is infinite for every element $g\in G$, which implies that $G$ coincides with $M$ and thus is virtually abelian, contradicting the fact that $G$ contains the acylindrically hyperbolic group $H$. Since $C$ is weakly malnormal, $G$ is acylindrically hyperbolic by \cite[Corollaries 2.2 and 2.3]{MO15}.

\smallskip

\textbf{Second case.} Suppose that $\lambda$ is finite. If $H$ were hyperbolic, one could prove that $\varphi_n$ is injective $\omega$-almost surely and conclude that $G=H$ (see Remark \ref{remarque} below). But this is not necessarily the case here. However, the sequence $(\varphi_n)_{n\in\mathbb{N}}$ gives rise to an action of $G$ on a $\delta$-hyperbolic space, namely the ultraproduct $(X_{\omega},d_{\omega})=\left(\prod_{n \in \mathbb{N}} (X,d)\right) / \omega$ without rescaling the metric. As observed in \cite{GH19}, Proposition 6.1, this action is acylindrical: let $\varepsilon >0$, let $R$ and $N$ be two constants given by the acylindrical action of $G$ on $X$, and let $(x_n)_{n\in\mathbb{N}}$, $(y_n)_{n\in\mathbb{N}}$ be two sequences of points of $X$ such that $d_{\omega}(x_{\omega}, y_{\omega})\geq R$. We claim that the set \[E=\lbrace g\in G \ \vert \ d_{\omega}(x_{\omega}, gx_{\omega})\leq \varepsilon \text{ and } d_{\omega}(y_{\omega}, gy_{\omega})\leq \varepsilon\rbrace\]has at most $N$ elements. Indeed, the inequalities $d_{\omega}(x_{\omega}, gx_{\omega})\leq \varepsilon$ and $d_{\omega}(y_{\omega}, gy_{\omega})\leq \varepsilon$ imply $d(x_n, \varphi_n(g)x_n)\leq \varepsilon$ and $d(y_n,\varphi_n(g)y_n)\leq \varepsilon$ $\omega$-almost surely, and it follows that $\varphi_n(E)$ has at most $N$ elements $\omega$-almost surely; since $(\varphi_n)_{n\in\mathbb{N}}$ is discriminating, one has $\vert E\vert \leq N$. Hence, the action of $G$ on the $\delta$-hyperbolic space $(X_{\omega},d_{\omega})$ is acylindrical. Of course, this result is interesting only if we can prove that this action is non-elementary. This is indeed the case. Recall that there exist two independent loxodromic elements $h_1$ and $h_2$ of $H$ such that $\varphi_n(h_1)=h_1$ and $\varphi_n(h_2)=h_2$ for every integer $n$. Therefore, the action of $\langle h_1,h_2\rangle\subset G$ on $X_{\omega}$ is non-elementary. Hence, the action of $G$ on $X_{\omega}$ is non-elementary, and $G$ is acylindrically hyperbolic.\end{proof}

\begin{rque}\label{remarque}The group $H$ being acylindrically hyperbolic, there exists a generating set $S$ of $H$ such that the Cayley graph $\Gamma(H,S)$ is hyperbolic and such that the natural action of $H$ on $\Gamma(H,S)$ is non-elementary and acylindrical, by \cite[Theorem 1.2]{Osi16}. If this generating set $S$ can be chosen finite (in other words, if $H$ is a hyperbolic group), then any non-divergent discriminating sequence $(\varphi_n : G \rightarrow H)_{n\in\mathbb{N}}$ is composed of injections $\omega$-almost surely, since $\varphi_n$ is completely determined by the image of a finite generating set of $G$ and since the graph $\Gamma(H,S)$ is locally finite.\end{rque}




\section{Applications}\label{section_coro}

In this section, we prove Corollary \ref{coro}. Recall that a group $G$ is called \emph{co-Hopfian} if any injective morphism $G\hookrightarrow G$ is bijective. In \cite{OH11}, Ould Houcine strengthened this definition, as follows.

\begin{de}A group $G$ is said to be \emph{strongly co-Hopfian} if there exists a finite subset $S\subset G\setminus \lbrace 1\rbrace$ such that, for any endomorphism $\phi$ of $G$, if $\ker(\phi)\cap S=\varnothing$ then $\phi$ is an automorphism.\end{de}

\begin{lemme}\label{prop}Let $G$ and $H$ be two finitely presented groups. Suppose that these groups are existentially equivalent. If $H$ is strongly co-Hopfian, then there exists an embedding $i : H \hookrightarrow G$ such that $i(H)$ is existentially closed in $G$.\end{lemme}

\begin{rque}This lemma remains true if $G$ is finitely generated and $H$ is equationally Noetherian, or if $H$ is finitely generated and $G$ is equationally Noetherian.\end{rque}

\begin{proof}Since $G$ and $H$ are existentially equivalent, and both are finitely presented, there exist two discriminating sequence $(\varphi_n : G \rightarrow H)_{n\in\mathbb{N}}$ and $(\psi_n : H \rightarrow G)_{n\in\mathbb{N}}$ (one can prove this fact exactly as in the proof of Lemma \ref{lemme1}). Let $S$ be the finite subset of $H\setminus\lbrace 1\rbrace$ given by the definition of the strongly co-Hopfian property. There exists an integer $n_0$ such that $\ker(\psi_{n_0})\cap S=\varnothing$. Then, for $n$ large enough, $\ker(\varphi_n\circ \psi_{n_0})\cap S=\varnothing$. As a consequence, $\varphi_n\circ \psi_{n_0}$ is an automorphism of $H$. In particular, $i:=\psi_{n_0}$ is injective. In addition, there exists a sequence $(\sigma_n)_{n\in\mathbb{N}}\in\mathrm{Aut}(H)^{\mathbb{N}}$ such that $\sigma_n\circ \varphi_n\circ i$ is the identity of $H$ for every integer $n$ large enough. If follows that $(i\circ \sigma_n\circ \varphi_n : G \rightarrow i(H))_{n\in\mathbb{N}}$ is a discriminating sequence of retractions. By Lemma \ref{lemme1}, $i(H)$ is existentially closed in $G$.\end{proof}

In order to prove Corollary \ref{coro}, it remains to explain why $\mathrm{Mod}(\Sigma_g)$, $\mathrm{Aut}(F_n)$, $\mathrm{Out}(F_n)$ and the Higman group are strongly co-Hopfian. This follows from the following facts.

\begin{itemize}
    \item[$\bullet$]For $g\geq 4$, any non-trivial endomorphism of $\mathrm{Mod}(\Sigma_g)$ is an automorphism (see \cite{AS12}, Corollary 1.4).
    \item[$\bullet$]For $n\geq 2$, any endomorphism of $\mathrm{Aut}(F_n)$ that is not an automorphism has finite image (see \cite{Khr5}).
    \item[$\bullet$]Any non-trivial endomorphism of the Higman group is an automorphism (see \cite{Mar17}, Theorem B).
    \item[$\bullet$]In \cite{BV0}, Bridson and Vogtmann proved that $\mathrm{Out}(\mathrm{Out}(F_n))$ is trivial for $n\geq 3$. Their proof contains the fact, non-explicitly stated, that $\mathrm{Out}(F_n)$ is strongly co-Hopfian for $n\geq 3$. We sketch a proof of this result below. Note that $\mathrm{Out}(F_2)$ is isomorphic to $\mathrm{GL}_2(\mathbb{Z})$, and one can prove that this group is strongly co-Hopfian; however, in this particular case, $\mathrm{Out}(F_2)$ is virtually free, thus hyperbolic, and it follows from \cite{And18} that any finitely generated group with the same $\forall\exists$-theory as $\mathrm{Out}(F_2)$ is hyperbolic.
\end{itemize}

\begin{te}[Bridson and Vogtmann]The group $\mathrm{Out}(F_n)$ is strongly co-Hopfian for $n\geq 3$.\end{te}

\begin{proof}
Let $\lbrace x_1,\ldots ,x_n\rbrace$ be an ordered generating set of $F_n$. For $i\in\llbracket 1,n\rrbracket$, let $e_i$ be the automorphism of $F_n$ that sends $x_i$ to $x_i^{-1}$ and fixes $x_j$ for $j\neq i$. For $i\in\llbracket 1,n-1\rrbracket$, let $\tau_i$ be the automorphism that interchanges $x_i$ and $x_{i+1}$ while leaving $x_j$ fixed for $j\notin \lbrace i,i+1\rbrace$. Let $W_n\simeq (\mathbb{Z}/2\mathbb{Z})^n\rtimes S_n$ be the finite subgroup of $\mathrm{Out}(F_n)$ generated by $e_i$ for $1\leq i\leq n$ and $\tau_i$ for $1\leq i\leq n-1$. The group $\mathrm{Out}(F_n)$ is generated by $W_n$ together with the involution $r$ that sends $x_1$ to $x_1x_2^{-1}$ and $x_2$ to $x_2^{-1}$ while leaving $x_i$ fixed for $i>2$. Let $G$ be the finite subgroup of $\mathrm{Out}(F_n)$ generated by $\lbrace r\rbrace\cup\lbrace \tau_1\rbrace\cup\lbrace e_i,\tau_i \ \vert \ 3\leq i\leq n\rbrace$. The group $\mathrm{Out}(F_n)$ is generated by $W_n\cup G$. The automorphism $u=re_2^{-1}$ maps $x_1$ to $x_1x_2$ and fixes $x_i$ if $i>2$. It follows that $u$ has infinite order. Define $S=(W_n\cup G \cup \lbrace u^{m!}\rbrace)\setminus \lbrace 1\rbrace$ where $m$ denotes the maximal order of a finite subgroup of $\mathrm{Out}(F_n)$ (note that this integer exists since every finite subgroup of $\mathrm{Out}(F_n)$ is isomorphic to the isometry group of a graph whose fundamental group is $F_n$). 

\smallskip

Let $\phi$ be an automorphism of $\mathrm{Out}(F_n)$ such that $\ker(\phi)\cap S=\varnothing$, and let us prove that $\phi$ is bijective. The subgroups of $\mathrm{Out}(F_n)$ that are isomorphic to $W_n$ are the stabilizers of the roses in outer space. Since the action of $\mathrm{Out}(F_n)$ is transitive on roses, each subgroup of $\mathrm{Out}(F_n)$ isomorphic to $W_n$ is conjugate to $W_n$. As a consequence, $\phi(W_n)$ is conjugate to $W_n$. Up to composing $\phi$ by an inner automorphism of $\mathrm{Out}(F_n)$, one can assume that $\phi(W_n)=W_n$. A calculation shows that the center of $W_n$ has order $2$; let $z$ be its non-trivial element. In \cite{BV0}, the authors prove by studying the action of $\phi(G)$ on the spine of outer space that $\phi(G)=G$ or $\phi(G)=G^z$. Up to composing $\phi$ with $\mathrm{ad}(z)$, one can assume that $\phi(G)=G$. Hence, one has $\phi(W_n)=W_n$ and $\phi(G)=G$. Since $W_n$ and $G$ are finite, there is a non-zero integer $k$ such that $\phi^k$ coincides with the identity on $W_n$ and on $G$. Since $\mathrm{Out}(F_n)$ is generated by $W_n\cup G$, $\phi^k$ is the identity. Therefore, $\phi$ is an automorphism.\end{proof}

\section{Counter-example among countable groups}\label{counter}

 If $G$ is any ultrapower of the free group $F_2$ with respect to a non-principal ultrafilter, then the centralizer of every element of $G$ is uncountable. Indeed, $(g_n)_{n\in\mathbb{N}}\in F_2^{\mathbb{N}}$ commutes with $(g_n^{k_n})_{n\in\mathbb{N}}$ for all sequences of integers $(k_n)_{n\in\mathbb{N}}$. This implies that $G$ is not acylindrically hyperbolic. Indeed, in an acylindrically hyperbolic group, every loxodromic element has virtually cyclic (and therefore countable) centralizer (see \cite[Lemma 6.5 and Corollary 6.6]{DGO17}). Since $G$ is elementarily equivalent to $F_2$ by \L{}o\'s theorem, this construction shows that acylindrical hyperbolicity is not preserved under elementary equivalence in general.
 
 \smallskip
 
 In fact, as observed by Osin, it is also not enough to restrict to countable groups. Indeed, let $G_0=F_2$ and let $g_1,g_2,\ldots$ be an enumeration of the non-trivial elements of $G_0$. Consider the following set of first-order formulas: $t_1(x)=\lbrace [x,g_1]=1, x\neq 1, x\neq g_1, x\neq g_1^2, \ldots\rbrace$. Since for every finite subset $t_1'(x)$ of $t_1(x)$ there exists an element $h\in G_0$ such that $G_0\models t_1'(h)$, by the compactness theorem there exists an elementary extension $G_0^1$ of $G_0$ and an element $h_1\in G_0^1$ such that $G_0^1\models t_1(h_1)$, i.e.\ $h_1$ commutes with $g_1$ but does not belong to $\langle g_1\rangle$. Iterating this operation, we get a chain of elementary extensions $G_0\prec G_0^1\prec G_0^2\prec \cdots$ such that for every $g_i\in G_0$, there exists an element $h_i\in G_0^i$ that commutes with $g_i$ but that does not belong to $\langle g_i\rangle$. Define $G_1=\cup_{i\in\mathbb{N}}G_0^i$. This group is an elementary extension of $G_0$. Continuing this process, we construct $G_0\prec G_1\prec G_2\prec \cdots$ and take the union $G'=\cup_{i\in\mathbb{N}}G_i$. Then $G'$ is a torsion-free, countable, elementary extension of $G_0$ that has no element with cyclic centralizer. Such a group cannot be acylindrically hyperbolic since the centralizer of a loxodromic element in a torsion-free acylindrically hyperbolic group is cyclic.
 
 \smallskip
 
 Thus to make Question \ref{osin} non-trivial, we have to ask it for finitely generated groups.

\renewcommand{\refname}{Bibliography}
\bibliographystyle{plain}
\bibliography{biblio}

\vspace{10mm}

\textbf{Simon André}

Vanderbilt University.

E-mail address: \href{mailto:simon.andre@vanderbilt.edu}{simon.andre@vanderbilt.edu}

\end{document}